\documentclass[reqno, 12pt]{amsart}
\usepackage{hyperref}
\usepackage{amsmath, amsthm, amssymb, mathtools, graphicx, paralist, color}
\usepackage{indentfirst}
\usepackage{appendix}
\usepackage{listings}
\usepackage{hyphenat}
\usepackage{calrsfs}
\usepackage{enumerate}
\usepackage{tikz-cd}
\usepackage{mathrsfs}
\usepackage{comment}
\hypersetup{
  hidelinks,
  pdftitle={Sign Bias in Hyperbolic Double Cosets},
  pdfauthor={Marios Voskou}
}
\newtheorem{conjecture}{Conjecture}

\newtheorem{theorem}{Theorem}
\newtheorem{lemma}{Lemma}
\theoremstyle{remark}
\newtheorem{remark}{Remark}
\usepackage[utf8]{inputenc}
\usepackage{listings}

\def \hGh {{\langle h \rangle \backslash \Gamma \slash \langle h \rangle}}
\def \psl  {{\hbox{PSL}_2( {\mathbb R})} }
\usepackage{bbm}
\usepackage[
backend=biber,
style=alphabetic,
sorting=nyt
]{biblatex}
\title{Sign bias in hyperbolic double cosets}
\author{Marios Voskou }
\address{Alfr\'{e}d R\'{e}nyi Institute of Mathematics, POB 127, Budapest H-1364, Hungary}
\email{voskou.marios@renyi.hu}
\date{\today}

\begin{document}
\begin{abstract} We study sign bias between hyperbolic double cosets with $ad>0$ and $ad<0$ in appropriate cocompact Fuchsian groups. We prove that an eventual bias can only favour $ad<0$, and show that such a bias occurs within every cocompact Fuchsian commensurability class. When the normal-derivative periods of all nonconstant eigenfunctions with eigenvalue at most $1/4$ vanish, we prove an $\Omega_\pm(X^{1/2})$ result and show that the counting difference takes both signs in every sufficiently large interval $[X,CX]$, for some fixed $C>1$.\end{abstract}
\maketitle
\section{Introduction}
Let $\Gamma$ be a cocompact Fuchsian group that contains a non-trivial primitive diagonal element $h %=\begin{pmatrix}
%m & 0\\
%0 & m^{-1}
%\end{pmatrix},
$ and no elements with both diagonal entries equal to $0$. Let $I$ be the positive imaginary axis. We study the double cosets $\langle h\rangle\gamma\langle h\rangle$ for which the distance between $I$ and $\gamma I$ is bounded by a growing parameter. This distance is related to the entries of $\gamma$ via the following equality, due to Martin--McKee--Wambach \cite[Lemma 1]{mmw}:
 for $\gamma= \begin{pmatrix}
a & b\\
c & d
\end{pmatrix},$
we have $$\cosh{\hbox{dist}(\gamma I, I) }= \mathrm{max}\left(\left| B(\gamma) \right|,1\right), \,
\hbox{where } \; \; B(\gamma):=ad+bc=2ad-1.$$
In particular, consider the function $$N(X):=\# \left\{\gamma \in \hGh \mid |ad|\leq X \right\}.$$
The problem of estimating $N(X)$ has been studied by various authors, including Good \cite{Good1983}, Martin--McKee--Wambach \cite{mmw}, %Parkkonen--Paulin \cite{ppaulin17},
and  Lekkas--Petridis \cite{lekkas2025}. 

We note that as $\Gamma$ is cocompact, the spectrum of the Laplacian on $\Gamma \backslash \mathbb{H}$ is discrete. We write $\lambda_j=s_j(1-s_j)$ for the eigenvalues, with multiplicity, and we fix a complete orthonormal basis of corresponding real-valued \(L^2\)-eigenfunctions $(u_j)_j$. We write $s_j=1/2+it_j$, where $t_j$ is either real or purely imaginary. For $l=\langle h\rangle\backslash I$, Lekkas--Petridis used spectral theory to prove that
\begin{equation*} N(X)=\frac{4\left(\mathrm{len}(l)\right)^2}{\pi\mathrm{Vol}\left(\Gamma \backslash \mathbb{H} \right)}X+\sum_{1/2 < s_j < 1}a_jX^{s_j}+O\left( X^{2/3} \right). \end{equation*}
%Here, the sum is over all the non-trivial  eigenvalues $\lambda_j=s_j(1-s_j)$ of the Laplacian on $\Gamma \backslash \mathbb{H}$ that are smaller than $1/4$, with multiplicity. %The explicit expressions for the coefficients $a_j$ can be found in \cite[Theorem 1.2]{lekkas2025}
The coefficients $a_j$ are obtained from \cite[Theorem 1.2]{lekkas2025} by rescaling the counting parameter by a factor of $2$. %We note that their result uses $|ad+bc|$ as the growth parameter rather than $|ad|$, but it is clear that the error between the two countings is no greater than $O(X^{2/3})$.
%We note that as the group is cocompact, the corresponding spectrum is discrete. We write $\lambda_j=s_j(1-s_j)$ for its eigenvalues, with multiplicity, and we fix a complete orthonormal basis of corresponding real-valued \(L^2\)-eigenfunctions $(u_j)_j$. We write $s_j=1/2+it_j$, where $t_j$ is either real or purely imaginary.

In this paper, we investigate potential bias between the double cosets with $ad>0$ and those with $ad<0$. Geometrically, for $|ad|>1$, the sign of $ad$ corresponds to whether $\gamma I$ is oriented horizontally away from or towards the imaginary axis (see \cite[Figure 1]{voskou2026}). We let
$$N^{\pm}(X)=\# \left\{\gamma \in \hGh \mid  |ad|\leq X, \; \; \hbox{sign}(ad)=\pm1\right\},$$ %and
%$$N^{-}(X)=\# \left\{\gamma \in \hGh  \mid  |ad|\leq X, \; \; ad<0\right\},$$
so that $N(X)=N^{+}(X)+N^{-}(X)$. We study the sign of $N_{\Delta}(X)=N^{+}(X)-N^{-}(X)$. %If there is an $X_0$ such that, for every $X>X_0$, we have $N_{\Delta}(X)>0$ (respectively, $<0$) we say that there is an \emph{eventual bias towards the positive (respectively, negative) sign of} $ad$.
A starting point for this investigation is the following theorem from \cite[Theorem 6]{voskou2026}, which states that the two counting functions are in fact asymptotically equal.
\begin{theorem} \label{refcount}
    For $N_{\Delta}(X)$ as above, we have the following asymptotic formula:
  $$N_{\Delta}(X)=\sum_{1/2<s_j < 1}-c(s_j)  \hat{u}_{1,j}^2X^{s_j}+O\left(X^{2/3}\right),$$
  where \begin{equation} \label{defcj} c(s)=\frac{\pi(1-s)\Gamma(2s-1)}
{4^{s-1}\Gamma^4(\frac{s+1}{2})},
%\frac{2^{s-1}(1-s)\Gamma(s-1/2)\Gamma(s/2)}{\Gamma^3((s+1)/2)},
  \end{equation}
  and, for $\lambda_j \neq 0$,
  $$\hat{u}_{1,j}:= \lambda_j^{-1/2}\int_{l} \partial_{\boldsymbol{n}}u_j ds,$$
  where $\boldsymbol{n}$ is a consistently chosen hyperbolic unit normal along the parametrised closed geodesic $l$, and $ds$ denotes hyperbolic arclength. In particular, $N^{+}(X) \sim N^{-}(X)$.
\end{theorem}
We also have the following conjecture regarding the best possible error term.
\begin{conjecture} \label{halfcon}
    The error term $O(X^{2/3})$ can be improved to $O_{\epsilon}(X^{1/2+\epsilon})$.
\end{conjecture}

 While the two signs contribute asymptotically the same number of double cosets, it is still possible to have eventual bias towards one of the two. As is evident from Theorem \ref{refcount}, one potential source of sign bias is the existence of a small eigenvalue $\lambda_j$ with $\hat{u}_{1,j} \neq 0$. In fact, we will prove that this is the only case in which $N_{\Delta}(X)$ can possibly have eventually constant sign. Furthermore, we show that while eventual bias towards the negative sign is possible, $N_\Delta(X)$ cannot be eventually nonnegative.

 \begin{theorem} \label{changes}
    Let $\Gamma$ be a cocompact Fuchsian group that contains a non-trivial diagonal element and no elements with both diagonal entries equal to $0$. Then:
\begin{itemize}
    \item[\textup{(a)}] If there is at least one $j$ with $s_j \in (2/3,1)$ and $\hat{u}_{1,j} \neq 0$, then there is an $X_0$ such that for every $X>X_0$, we have $N_{\Delta}(X)<0$.
    \item[\textup{(b)}] If there are no values of $j$ with $s_j \in [1/2,1)$ and $\hat{u}_{1,j} \neq 0$, then %$N_{\Delta}(X)$ changes sign infinitely many times as $X \rightarrow \infty$.
    there exist constants $X_0$ and $C>1$ such that, for every $X>X_0$, there are $x_{+},x_{-} \in [X,CX]$ with $$ N_{\Delta}(x_{+})> \sqrt{X} \quad \hbox{  and  } \quad N_{\Delta}(x_{-})<-\sqrt{X}.$$ 
    \item[\textup{(c)}] There is no $X_0$ such that $N_{\Delta}(X) \geq 0$ for every $X> X_0$.
\end{itemize}
In particular, $N_{\Delta}(X)$ either eventually stays nonpositive, or changes sign infinitely many times.
\end{theorem}
\begin{remark}
    Under the condition of part (b), it is easy to deduce that  $N_{\Delta}(X)=\Omega_{\pm}(X^{1/2})$, and that there are $\gg \log{X}$ sign changes up to $X$. The $\Omega_{\pm}$ result can also be proved using methods similar to \cite{VoskouOmega}.
\end{remark}
\begin{remark}
     Under Conjecture \ref{halfcon}, we obtain the stronger conclusion that if $\lambda_j<1/4$ and $\hat{u}_{1,j} \neq 0$ for some $j$, then there is an eventual bias towards the negative sign. However, even under Conjecture \ref{halfcon}, it is not clear how the sign behaves if the only  eigenvalue $\lambda_j=s_j(1-s_j)$ at most $1/4$ and with nonzero corresponding period is $\lambda_j=1/4$. %We believe that one can use our work in \cite{lekkas2024} to show that in that case $N_{\Delta}(X)$ is negative on a set of logarithmic density one, but it is not clear whether this can be strengthened to hold for all sufficiently large $X$. Under Conjecture \ref{halfcon}, the above discussion covers every other case.
\end{remark}

 Part (a) follows immediately from Theorem \ref{refcount}, since $c(s)>0$ for every $s \in (1/2,1)$. For parts (b) and (c), we study a normalised average of $N_\Delta$. Its spectral expansion separates the contributions of small eigenvalues from a uniformly convergent oscillatory series. When the former are absent, estimates for the first and second moments of the average force sign changes. This argument is a continuous analogue of a method of Meher and Murty \cite[Theorem 1.1]{murphy}.
 %In the absence of small eigenvalues with nonzero periods, its uniformly convergent spectral expansion yields first- and second-moment estimates uniform over moving intervals. Comparing these moments forces positive and negative values of the average, which we then transfer to $N_\Delta$. We use a continuous analogue of the method developed by Meher and Murty in \cite[Theorem 1.1]{murphy}.  
 Related uses of averaging to obtain $\Omega$-results in hyperbolic lattice point problems appear in the work of Phillips--Rudnick \cite{philipsr} and Chatzakos \cite{chatzakos, cha19}.

We also construct groups satisfying the conditions of Theorem \ref{changes}(a), obtaining the following result:
 \begin{theorem} \label{bias}

 For every $\epsilon>0$ and cocompact Fuchsian group $\tilde{\Gamma}$, there is a positive $\delta< \epsilon$, an appropriate group $\Gamma$ that is conjugate to a normal subgroup of $\tilde{\Gamma}$, and a constant $k>0$ such that
     $$N^{-}(X)-N^{+}(X) \sim kX^{1-\delta}.$$
In particular, every cocompact Fuchsian commensurability class contains some group $\Gamma$ such that there is some real number $X_0$ with $N^{-}(X)>N^{+}(X)$ for every $X>X_0$.
 \end{theorem}
 \begin{remark}
     We expect Theorem 3 to hold for arbitrary cofinite Fuchsian groups, although we do not pursue this extension here.
 \end{remark}
 The proof combines a construction of Cheeger \cite{Cheeger+1971+195+200} with a variational argument involving the first positive Neumann eigenvalue to establish the nonvanishing of an appropriate period.

\section{Asymptotics for a normalised average of $N_{\Delta}$}
For $X \geq 1$, define
$$\overline{N_{\Delta}}(\log{X}):=X^{-3/2}\int_{X}^{3X/2}N_{\Delta}(u) du.$$ 
This is a normalised average of $N_{\Delta}(u)$ over the interval $[X,3X/2]$.
A key ingredient in the proof of Theorem \ref{changes} is the following result, providing asymptotics for $\overline{N_{\Delta}}(y)$. In particular, if there are no small eigenvalues with nonzero corresponding periods, the function $\overline{N_{\Delta}}(y)$ can be approximated by an almost-periodic function. For related results on almost periodicity and asymptotic moments of integrated hyperbolic counting errors, see \cite{cherubini}.
\begin{theorem} \label{almper}
    For $\overline{N_{\Delta}}(y)$ as above, we have
    $$\overline{N_{\Delta}}(y)=- \!\!\!\!\sum_{1/2<s_j<1}
\alpha(s_j)\hat u_{1,j}^2e^{(s_j-1/2)y}
-\kappa \left(\sum_{t_j=0}\hat{u}_{1,j}^2\right)(y-\eta)+\textup{Re}\sum_{t_j>0}b(t_j)\hat{u}_{1,j}^2e^{it_jy}+o(1),$$
    where $\kappa$ and $\eta$ are positive constants,
    \begin{equation} \label{coefb} \alpha(s)=\frac{c(s)}{s+1}
\left[\left(\frac32\right)^{s+1}-1\right], \qquad 
b(t)=-2\alpha\left(\frac{1}{2}+it\right)
%-\left(\frac{1}{4}+t^2\right)\left(3^{3/2+it}-2^{3/2+it}\right)\gamma(t)=-2\alpha(s)
,\end{equation}
and $c(s)$ is defined as in equation (\ref{defcj}).
%and, for $s=1/2+it$, we have
%$$\gamma(t):=\frac{\left(\Gamma\left(s/2\right)\right)^2\Gamma\left(s-1/2\right)}{8\left(\Gamma\left((1+s)/2\right)\right)^2 \Gamma\left(1+s/2\right) \Gamma\left((3+s)/2\right)}.$$
Furthermore, \begin{equation} \label{best}\sum_{t_j>0}|b(t_j)|\hat{u}_{1,j}^2< \infty . \end{equation}
\end{theorem}

\begin{remark} \label{thm4c} Observe that the coefficients $\alpha(s_j)$ and $\kappa$ are positive. Therefore, if there is at least one small eigenvalue $\lambda_j \leq 1/4$ with $\hat{u}_{1,j} \neq 0$, the quantity $\overline{N_{\Delta}}(y)$ will be negative for $y$ large enough. Hence, in this case, $N_{\Delta}(X)$ cannot be nonnegative for all large enough $X$.  In light of this, to prove Theorem \ref{changes}(c) it is enough to prove Theorem \ref{changes}(b).
\end{remark}
To prove Theorem \ref{almper}, recall the following result from \cite[Theorem 7(c)]{voskou2026}.

\begin{theorem}[Relative Trace Formula] \label{modtrace} Let $\Gamma$ be a cocompact Fuchsian group with at least one non-trivial primitive diagonal element $h$, and no element with both diagonal entries zero. Let $f:[1,\infty) \rightarrow \mathbb{R}$ be a real, continuous, piecewise differentiable function with exponential decay. Then, we have
\begin{equation}\label{trace}
 \; \displaystyle  f(1)\mathrm{len}(l)+\sum_{\gamma \in \hGh-\mathrm{id}} B(\gamma)g\left(B^2(\gamma); f\right)=2\sum_{j} \lambda_jd^{(1)}_{t_j}(f)\hat{u}^2_{1,j},
\end{equation}
where $B(\gamma)= ad+bc$,
 $\mathrm{len}(l)$ is the hyperbolic length of $l$,

\begin{equation*} g(u;f):=2\int_{ \sqrt{\mathrm{max}(u-1,0)}}^{\infty}  \frac{\left(xf\left( x^2+1  \right)\right)'}{\sqrt{x^2+1-u}}dx \end{equation*}

and, for $s=1/2+it$,
\begin{equation*} d^{(1)}_{t}(f):=\int_{0}^{\pi/2}\frac{\tan^2{v}}{\cos^2{v}} f\left(\frac{1}{\cos^2{v}}\right) \cdot {}_2F_1\left(\frac{s+1}{2},\frac{2-s}{2} \, ; \frac{3}{2} \, ; -\tan^{2}{v} \right)\, dv.\end{equation*}
\end{theorem}

We apply Theorem \ref{modtrace} for $f=f_X=f_4(\,\cdot\,;X,D=1/2)$, where $f_4$ and $D$ are as in \cite[Section 5]{voskou2026}. We write $d_{t}^{(1)}(X):=d_t^{(1)}(f_X)$. 
As  in \cite[Section 5]{voskou2026}, we will have

$$ g\left(u^2+1;f\right) = \left\{
\begin{array}{lr}
      \displaystyle (pu^2+q), & 0< u \leq \sqrt{2}, \\
     
     \displaystyle \frac{u^{-1}}{R-r}\int_{r}^{R} \mathbbm{1}(u<v) \, dv, & \phantom{0}\phantom{<}u>\sqrt{2}, \\
\end{array} 
\right. $$
where $R=\sqrt{9X^2/4-1}, r=\sqrt{X^2-1}$ and
\begin{equation*}p 
= \frac{3\sqrt{2}}{8}\left( \frac{R\log{R}-r\log{r}}{R-r}-\frac{\log{2}}{2}\right), \qquad
q=\frac{1}{\sqrt{2}}-2p.
\end{equation*}

The contribution of $|B(\gamma)| \leq \sqrt{3}$ in the left-hand side of equation (\ref{trace}) is clearly $O(p)=O(\log{X})$. The contribution of $|B(\gamma)|>\sqrt{3}$ becomes
$$\frac{1}{R-r}\sum_{|B(\gamma)|>\sqrt{3}}\frac{B(\gamma)}{\sqrt{B^2(\gamma)-1}}\int_{r}^{R} \mathbbm{1}\left(\sqrt{B^2(\gamma)-1}<v\right) \, dv.$$
Recall now that, for $|x| >\sqrt{3}$, we have
$$
\frac{x}{\sqrt{x^2-1}}
=\operatorname{sgn}(x)+O(x^{-2}).
$$
Hence, we have
\begin{align*}&\sum_{|B(\gamma)|>\sqrt{3}}\frac{B(\gamma)}{\sqrt{B^2(\gamma)-1}}\int_{r}^{R} \mathbbm{1}\left(\sqrt{B^2(\gamma)-1}<v\right) \, dv \\
&=\sum_{|B(\gamma)|>\sqrt{3}} \hbox{sign}(ad)\int_{r}^{R} \mathbbm{1}\left(\sqrt{B^2(\gamma)-1}<v\right) \, dv +O(X) \\
&=\sum_{|B(\gamma)|>\sqrt{3}}\hbox{sign}(ad)\int_{X}^{3X/2} \mathbbm{1}\left(|B(\gamma)|<t\right) \frac{t \, dt}{\sqrt{t^2-1}} +O(X) \\
&=\sum_{|B(\gamma)|>\sqrt{3}}\hbox{sign}(ad)\int_{X}^{3X/2} \mathbbm{1}\left(|B(\gamma)|<t\right) \, dt +O(X\log{X}), \end{align*}
 where, for the last equality, we used 
 that 
 \begin{align*}\sum_{|B(\gamma)|>\sqrt{3}}\hbox{sign}(ad)\int_{X}^{3X/2} \mathbbm{1}\left(|B(\gamma)|<t\right) \left( \frac{t}{\sqrt{t^2-1}}-1\right) dt \ll N(X)\int_{X}^{3X/2} t^{-1} dt \ll X\log{X}. 
\end{align*}
To change the growth parameter from $|B(\gamma)|$ to $|ad|$, note that
\begin{align*}
&\sum_{|B(\gamma)|>\sqrt{3}}\hbox{sign}(ad)\int_{X}^{3X/2} \mathbbm{1}\left(|B(\gamma)|<t\right) \, dt\\
&=\sum_{|B(\gamma)|>\sqrt{3}}\hbox{sign}(ad)\int_{X}^{3X/2} \mathbbm{1}\left(2|ad|<t\right) \, dt + O\left(\sum_{|B(\gamma)|<3X/2+1}\int_{X}^{3X/2} \mathbbm{1}\left(|t-2|ad||<1\right) \, dt\right)  \\
%&= \int_{X}^{3X/2} N_{\Delta}(t/2) \frac{t}{\sqrt{t^2-1}} \, dt + O\left(X\right) \\
&= \int_{X}^{3X/2} N_{\Delta}(t/2) \, dt + O\left(X\right) \\
&= 2\int_{X/2}^{3X/4} N_{\Delta}(u) \, du + O\left(X\right).
\end{align*}

Hence, we have
$$\frac{2}{R-r}\int_{X/2}^{3X/4}N_{\Delta}(u)du=2\sum_{j} \lambda_jd^{(1)}_{t_j}( X)\hat{u}^2_{1,j}+O(\log{X}).$$
As $R-r=X/2+O(X^{-1})$,
it follows that
$$\overline{N_{\Delta}}(\log{X})=X^{-3/2}\int_{X}^{3X/2}N_{\Delta}(u)du=X^{-1/2}\sum_{j} \lambda_jd^{(1)}_{t_j}( 2X)\hat{u}^2_{1,j}+O(X^{-1/2}\log{X}).$$

We will now prove the following estimates for the coefficients $d^{(1)}_{t_j}(X)$ and the coefficients $b(t_j)$.
\begin{lemma} \label{estimates}
 The following estimates hold:
    \begin{itemize}
  
    \item[\textup{(a)}] For fixed $s=1/2+it$ with $t>0$, we have \begin{align*}s(1-s)d^{(1)}_{t}(2X) X^{-1/2}%&=b(t)X^{it}+b(-t)X^{-it}+O_{t}\left(X^{-1/2}\log{X} \right) \\
    %&
    = \textup{Re}\left(b(t)X^{it}\right)+O_{t}\left(X^{-1/2}\log{X} \right).\end{align*}
     \item[\textup{(b)}] For fixed $s_j \in (1/2,1)$, we have $$s_j(1-s_j) d^{(1)}_{t_j}(2X) X^{-1/2}=-\alpha(s_j)X^{s_j-1/2}
     %-\alpha(1-s_j)X^{1/2-s_j}
     +o_{s_j}\left(1 \right).$$
   \item[\textup{(c)}] For $s_j=1/2$, we have $$\frac{1}{4} d^{(1)}_{0}(2X) X^{-1/2}=-\kappa\left(\log{X}-\eta\right)+O\left(X^{-1/2}\log{X} \right),$$
   where $\kappa$ and $\eta$ are positive constants.
    \item[\textup{(d)}] For $T>1$ and $X$ large enough,
    $$ X^{-1/2}\sum_{t_j>T} \lambda_j\left|d^{(1)}_{t_j}(2X)\right|\hat{u}^2_{1,j} \ll T^{-1/2}.$$
    \item[\textup{(e)}] We have $b(t) \ll t^{-3/2}$, and therefore $$\sum_{t_j>T}|b(t_j)|\hat u_{1,j}^2 \ll T^{-1/2}.$$
\end{itemize}
\end{lemma}
\begin{proof}[Proof of Lemma \ref{estimates}]
First, note that using \cite[equation (32) and Lemma 3(b)]{voskou2026}, and following the argument in the proof of \cite[Lemma 8(b)]{voskou2026}, we have that \begin{equation} \label{compactest}s(1-s)d^{(1)}_{t}(2X) X^{-1/2}
=-\alpha(s)X^{s-1/2}
-\alpha(1-s)X^{1/2-s}
+O_{\mathcal S}(X^{-1/2}\log X),
\end{equation}
uniformly for $s$ in any compact subset
$
\mathcal S\subset\left\{s\in\mathbb C:0<\hbox{Re} (s)<1,\ s\ne\frac{1}{2}\right\}.
$
The uniformity follows from the locally uniform convergence of the hypergeometric series occurring in the connection formula. Part (a) and part (b) then follow immediately.
    
    For part (c), we first calculate the Laurent expansion of $\alpha(1/2+z)$. Using the definition of $\alpha$ from equation (\ref{coefb}), we have
    $$\alpha\left( \frac{1}{2}+z \right)=\frac{\kappa}{2z}-\frac{\kappa\eta}{2}+O(z),
$$
where $$
\kappa=\frac{\Gamma(1/4)(3^{3/2}-2^{3/2})}
{6\Gamma(3/4)^3}>0, \quad \hbox{and} \quad
\eta=\pi+\frac{8}{3}-4\log2-\frac{3^{3/2}\log(3/2)}{3^{3/2}-2^{3/2}}>0.
$$

Hence, for a fixed $X$, the expression $$\alpha\left( \frac{1}{2}+z \right)X^{z}+\alpha\left( \frac{1}{2}-z \right)X^{-z}$$ is holomorphic near $z=0$. 

But $\left(\frac{1}{4}-z^2\right)d^{(1)}_{-iz}(2X) X^{-1/2}$ is also holomorphic by the definition of $d^{(1)}$. Therefore,
the function $$F(z)=\left(\frac{1}{4}-z^2\right)d^{(1)}_{-iz}(2X) X^{-1/2}+\alpha\left( \frac{1}{2}+z \right)X^{z}+\alpha\left( \frac{1}{2}-z \right)X^{-z} $$ is holomorphic near $z=0$, which using the Laurent expansion above gives, for $z=0$, 
$$
F(0)
=\frac14d_0^{(1)}(2X)X^{-1/2}
+\kappa(\log X-\eta).
$$
By the maximum modulus principle in the region $|z|< \epsilon$ for some fixed, small enough $\epsilon$, we have 
$|F(0)| \ll \sup_{|z|=\epsilon}|F(z)|.$ On the other hand, by equation (\ref{compactest}), we have $\sup_{|z|=\epsilon}|F(z)| \ll X^{-1/2}\log{X}$. Hence, $F(0)=O(X^{-1/2}\log X)$, which finishes the proof of part (c).

    Part (d) follows from \cite[Lemma 6(b)]{voskou2026} combined with \cite[Lemma 13]{voskou2026} by a dyadic decomposition.
    Finally, for part (e), we apply Stirling's approximation formula to get $b(t) \ll t^{-3/2}$, and then we proceed as in part (d).
\end{proof}

We are ready to prove Theorem \ref{almper}.
\begin{proof}[Proof of Theorem \ref{almper}]
   
Consider the quantity
$$D=\limsup_{X \rightarrow \infty}\left| X^{-1/2}\sum_{t_j>0} \lambda_jd^{(1)}_{t_j}(2X)\hat{u}^2_{1,j} - \textup{Re}\sum_{t_j>0}b(t_j)\hat u_{1,j}^2X^{it_j}\right|. $$
By Lemma \ref{estimates}(a),(d),(e) we can see that, for every $T>1$, we have
$D \ll T^{-1/2}.$ Taking $T \rightarrow \infty$, we have $D=0$. By Lemma \ref{estimates}(b),(c), the contribution of the eigenvalues with $s_j \in [1/2,1)$ will tend to $$-\sum_{1/2<s_j<1}
\alpha(s_j)\hat u_{1,j}^2X^{s_j-1/2}
-\kappa \left(\sum_{t_j=0}\hat{u}_{1,j}^2\right)(\log{X}-\eta).$$
By combining the above, the first assertion of the theorem follows.

The inequality $$\sum_{t_j>0}|b(t_j)|\hat{u}_{1,j}^2< \infty$$ follows directly from Lemma \ref{estimates}(e).
\end{proof}

\section{Proof of Theorem \ref{changes}}
Meher and Murty \cite[Theorem 1.1]{murphy} give a criterion
for sign changes of a real sequence based on estimates for its first and second
moments. We prove Theorem \ref{changes}(b) using a
continuous analogue of their argument.
To this end, we study the asymptotic first and second moments of the function $\overline{N_{\Delta}}(y)$.
\begin{lemma}\label{asmom}
    Suppose that $\hat{u}_{1,j}=0$ for every $j$ with
    $0<\lambda_j\leq 1/4$. Then
    $\overline{N_{\Delta}}$ is bounded on $[0,\infty)$,
    and there exists a constant $L>0$ such that
    $$
        \lim_{H\to\infty}\sup_{Y\geq 0}
        \left|
            \frac{1}{H}\int_Y^{Y+H}
            \overline{N_{\Delta}}(y)\,dy
        \right|=0
    \quad
    \hbox{and}
    \quad
        \lim_{H\to\infty}\sup_{Y\geq 0}
        \left|
            \frac{1}{H}\int_Y^{Y+H}
            \overline{N_{\Delta}}(y)^2\,dy-L
        \right|=0.
    $$
\end{lemma}
\begin{proof}
    Write
    $$
A(y):=\hbox{Re}\sum_{t_j>0}b(t_j)\hat{u}_{1,j}^{\,2}e^{it_jy}.
    $$
    By Theorem \ref{almper} and the hypothesis on the small eigenvalues,
    $$
        \overline{N_{\Delta}}(y)=A(y)+r(y),
        \qquad r(y)=o(1).
    $$
    Moreover, by equation (\ref{best}), the series defining $A$ converges absolutely and
    uniformly on $\mathbb R$. In particular, $A$ is bounded.
    It follows
    that $\overline{N_{\Delta}}$ is also bounded.

    We first establish the required uniform moment limits
    for $A$. Grouping terms with the same spectral parameter,
    put
    $$
        a_t:=b(t)\sum_{t_j=t}\hat{u}_{1,j}^{\,2},
    $$
    where $t$ ranges over the distinct positive spectral
    parameters. Then
    $$
        A(y)=\hbox{Re}\sum_{t>0}
        a_te^{ity},
        \qquad
        \sum_{t>0}|a_t|<\infty.
    $$
    For every nonzero real number $\omega$,
    $$
        \left|
            \frac1H\int_Y^{Y+H}e^{i\omega y}\,dy
        \right|
        =
        \left|
            \frac{e^{i\omega Y}(e^{i\omega H}-1)}
                 {i\omega H}
        \right|
        \leq \frac{2}{H|\omega|}.
    $$
    Thus the average of each nonconstant exponential tends
    to zero uniformly in $Y$.
Therefore,
     $$
        \lim_{H\to\infty}\sup_{Y\in\mathbb R}
        \left|
            \frac1H\int_Y^{Y+H}A(y)\,dy
        \right|=0.
    $$
    Indeed, truncate the series so that the sum of the
    absolute values of the omitted coefficients is less
    than $\epsilon$.  The finitely
    many remaining nonconstant terms tend uniformly to zero
    by the preceding estimate.

    Similarly, the double series obtained by expanding $A(y)^2$ is
    also absolutely and uniformly convergent, and we have
    $$
        \lim_{H\to\infty}\sup_{Y\in\mathbb R}
        \left|
            \frac1H\int_Y^{Y+H}A(y)^2\,dy-L
        \right|=0,
    $$
    where
    \begin{equation}\label{ldef}
        L:=\frac12\sum_{t>0}|a_t|^2
        =\frac12\sum_{t>0}|b(t)|^2
        \left(\sum_{t_j=t}\hat{u}_{1,j}^{\,2}\right)^2.
    \end{equation}

    We next transfer these limits from $A$ to $\overline{N_{\Delta}}$.
    Given $\epsilon>0$, choose $R>0$ such that
    $|r(y)|\leq\epsilon$ for $y\geq R$. For every
    $Y\geq0$,
    $$
        \frac1H\int_Y^{Y+H}|r(y)|\,dy
        \leq
        \frac1H\int_0^R|r(y)|\,dy+\epsilon.
    $$
    Consequently,
    $$
        \lim_{H\to\infty}\sup_{Y\geq0}
        \frac1H\int_Y^{Y+H}|r(y)|\,dy=0.
    $$
    Furthermore, since $\overline{N_{\Delta}}(y)$ and $A$ are bounded on $[0,\infty)$, there
    is a constant $K>0$ such that
    $$
        |\overline{N_{\Delta}}(y)^2-A(y)^2|
        =|r(y)|\,|\overline{N_{\Delta}}(y)+A(y)|
        \leq K|r(y)|.
    $$
    The asserted uniform limits for $\overline{N_{\Delta}}$ and $\overline{N_{\Delta}}^2$
    now follow from those for $A$ and $A^2$.
    
 It remains to show that $L>0$. It is enough to show that at least one of the summands in equation (\ref{ldef}) is not zero. To show this, note that $b(t) \neq 0$ for all $t>0$, and hence we need $\hat{u}_{1,j} \neq 0$ for at least one $j$. If this is not the case, the right-hand side of Theorem \ref{modtrace} vanishes, meaning that the left-hand side must also vanish for any choice of $f$. However, it is clear that this is not the case, as we can choose $f$ to be a function supported in an interval $[1,1+\epsilon]$, and with $f(1)=1$. This would make the left-hand side equal to $f(1) \hbox{len}(l)+o(1)$, which is nonzero for $\epsilon$ small enough, contradicting our assumption. Hence, $L>0$.
   
\end{proof}

We now finish the proof of Theorem \ref{changes}(b) and, therefore, of Theorem \ref{changes}(c) (recall the reduction from Remark \ref{thm4c}).

\begin{proof}[Proof of Theorem \ref{changes}(b),(c)]
    Assume the hypotheses of part (b).
    By Lemma \ref{asmom}, there are constants $M>0$
    and $L>0$ such that $|\overline{N_{\Delta}}(y)|\leq M$ for $y\geq0$,
    and the first two moments converge uniformly over
    moving intervals to $0$ and $L$, respectively.
    Hence we may fix $H>0$ sufficiently large so that,
    for every $Y\geq0$,
    $$
        \left|\frac1H\int_Y^{Y+H}\overline{N_{\Delta}}(y)\,dy\right|
        <\frac{L}{4M},
        \qquad
        \frac1H\int_Y^{Y+H}\overline{N_{\Delta}}(y)^2\,dy>\frac L2.
    $$
    Since $\overline{N_{\Delta}}(y)^2\leq M|\overline{N_{\Delta}}(y)|$, it follows that
    $$
        \frac1H\int_Y^{Y+H}|\overline{N_{\Delta}}(y)|\,dy>\frac{L}{2M}.
    $$
   % Put $\overline{N_{\Delta}^{\pm}}(y):=\max\{\pm\overline{N_{\Delta}}(y),0\}=(|\overline{N_{\Delta}}|\pm \overline{N_{\Delta}})/2.$ We have
  %  $$
  %      \frac{1}{H}\int_Y^{Y+H}\overline{N_{\Delta}^{\pm}}(y)\,dy
  %      >\frac{L}{8M}.
   % $$
   Therefore,
    $$    \frac{1}{H}\int_Y^{Y+H}\left| \overline{N_{\Delta}}(y)\,dy\right|\pm\overline{N_{\Delta}}(y)\,dy
        >\frac{L}{4M}.
    $$
    Thus, with $\delta:=L/(8M)$, every interval
    $[Y,Y+H]$ contains points $y_+$ and $y_-$ such that
    $$
        \overline{N_{\Delta}}(y_+)>\delta,
        \qquad \overline{N_{\Delta}}(y_-)<-\delta.
    $$
Equivalently, for every $X >1$, the interval $[X,e^{H}X]$ contains points $U_{\pm}$ such that
$$
        \int_{U_+}^{3U_+/2}N_\Delta(u)\,du
        >\delta U_+^{3/2},
        \qquad
        \int_{U_-}^{3U_-/2}N_\Delta(u)\,du
        <-\delta U_-^{3/2}.
    $$
  %  Let $Z\geq1$, take $Y=\log Z$, and put
    %$U_\pm=e^{y_\pm}$. Then
  %  $Z\leq U_\pm\leq e^H Z$, and the definition of $\overline{N_{\Delta}}$
  %  gives
  %  $$
        %\int_{U_+}^{3U_+/2}N_\Delta(u)\,du
       % >\delta U_+^{3/2},
       % \qquad
      %  \int_{U_-}^{3U_-/2}N_\Delta(u)\,du
     %   <-\delta U_-^{3/2}.
    %$$
    Consequently, there are points
    $x_\pm\in[U_\pm,3U_\pm/2] \subset [X,2e^{H}X]$ satisfying
    $$
        N_\Delta(x_+)>\delta\sqrt{U_+}
        \geq\delta\sqrt X,
        \qquad
        N_\Delta(x_-)<-\delta\sqrt{U_-}
        \leq-\delta\sqrt X.
    $$
  %  In particular, both points lie in
  %  $[Z,C_0Z]$, where $C_0:=3e^H/2$.
%
   % Finally, set
   % $$
   %     K:=\max\{1,\delta^{-2}\},
   %     \qquad C:=KC_0,
   % $$
   % and apply the preceding conclusion with $Z=KX$.
   % Since $[KX,C_0KX]\subseteq[X,CX]$, we obtain
   % points $x_+,x_-\in[X,CX]$ such that
   % $$
   %     %N_\Delta(x_+)>\sqrt X,
       % \qquad N_\Delta(x_-)<-\sqrt X.
   % $$
   % This proves part (b).
%
    %Since $x_\pm\leq CX$, these inequalities also give
%$N_\Delta(X)=\Omega_\pm(X^{1/2})$.
    %Applying the assertion to the disjoint intervals
    %$$
        %[X_*C^{2k},X_*C^{2k+1}],
      %  \qquad k=0,1,2,\ldots,
    %$$
    %with $X_*$ sufficiently large, gives
    %$\gg\log X$ sign changes up to $X$.
    Part (b) now follows from rescaling. 
    
    Part (c) follows from part (b) and the contribution
    of the small eigenvalues, as explained in Remark \ref{thm4c}.
\end{proof}

\section{Proof of Theorem \ref{bias}}

We use the construction from \cite{Cheeger+1971+195+200} (see also \cite[Theorem 25]{Burrin2021} and \cite[Chapter 8]{Bus92}) to find, for a fixed $\epsilon>0$, a group $\Gamma(\epsilon)$ with at least one eigenvalue $\lambda=s(1-s)$ with $s > 1-\epsilon$ and at least one associated eigenfunction with nonvanishing normal-derivative period. %In fact, it is enough to show that $\hat{u}_{1,1,\epsilon}/\epsilon$ does not tend to $0$ (unfortunately we do know that $\hat{u}_{1,1,\epsilon} \rightarrow 0$).
Fix a cocompact Fuchsian group $\tilde{\Gamma}$. We will show that $\Gamma(\epsilon)$ can be chosen to be the conjugate of a normal subgroup of $\tilde{\Gamma}$, with large enough index.

By a lemma of Selberg \cite[Lemma 8]{selberg2}, $\tilde{\Gamma}$ has a torsion-free normal subgroup $\Gamma$ of finite index.
Then, $M=\Gamma\backslash\mathbb {H}$ is a closed hyperbolic surface of genus at least two. Let $ \mathcal{C}\subset M $ be an oriented nonseparating simple closed geodesic, corresponding to an element $h$ of $\Gamma$. 

We cut $M$ along $\mathcal{C}$ to obtain a surface $M'$ with $\partial M'=\mathcal{C}^+ \cup \mathcal{C}^{-}$, where $\mathcal{C}^{\pm}$ are two copies of $\mathcal{C}$ with opposite orientation. We consider a family of copies of $M'$, denoted by $(M'_i)_i$. We write $\mathcal{C}^{\pm}_i$ for the corresponding copies of the curves $\mathcal{C}^{\pm}$. For every $n$, let $M_{n}$ be the surface we get by gluing together $M'_1,\ldots,M'_n$ by identifying $\mathcal{C}^+_i$ with $\mathcal{C}^-_{i+1}$ for every $1 \leq i<n$, and $\mathcal{C}^+_n$ with $\mathcal{C}^-_{1}$. For every $n$, there is a cocompact Fuchsian group $\Gamma_n$ which is a normal subgroup of $\Gamma$, and $M_n=\Gamma_n\backslash\mathbb {H}$. Note that $h \in \Gamma_{n}$ for every $n$. In fact, $\Gamma_n$ contains all the conjugates of $h$ in $\Gamma$. Indeed, $\Gamma_n=\hbox{ker}(\phi)$, where $\phi: \Gamma \rightarrow \mathbb{Z} \slash  n \mathbb{Z}$ is the algebraic intersection number with $\mathcal{C}$ modulo $n$. We note that since $\mathcal{C}$ is nonseparating, $\phi$ is surjective, and hence $[\Gamma:\Gamma_n]=n$. The $n$ lifts of $\mathcal{C}$ correspond to $n$ distinct $\Gamma_n$-conjugacy classes, represented by some $\Gamma$-conjugates $h_1,\ldots,h_n$ of $h$. We order these elements so that $h_i$ corresponds to the curve $\mathcal{C}_i=\mathcal{C}_i^{+}=\mathcal{C}_{i+1}^{-}$, where $\mathcal{C}_{n+1}^{-}=\mathcal{C}_{1}^{-}$.

Let $\lambda_1(M_n)$ be the smallest non-trivial eigenvalue of the Laplacian over $M_n$.  As in \cite{Cheeger+1971+195+200}, we have $\lambda_1(M_{2n}) \rightarrow 0$.  We denote the smallest non-trivial Neumann eigenvalue for $M'$ by $\mu$. This is well defined, as Neumann eigenvalues for a compact connected Riemann surface with smooth boundary are discrete and nonnegative (see \cite[Chapter I, §3, Theorem 1]{chavel}). Take $n$ such that $s_1(M_{2n})>\hbox{max}(1-\epsilon,2/3)$ and $\lambda_1(M_{2n})<\mu$. Let $u=u_1$ be the corresponding eigenfunction.

For every $i$, let $\boldsymbol{n}_i$ be the unit normal vector field along $\mathcal{C}_i$, pointing from $M_i'$ into $M_{i+1}'$, where we use the convention $M'_{2n+1}=M'_{1}$. We will show that there is some $i$ such that
the integral $\hat{u}_{\mathcal{C}_i}=\int_{\mathcal{C}_i}\partial_{\boldsymbol{n}_i}u \, ds$ is nonzero.

It is enough to show that there is some $i$ such that $\int_{M'_{i}} u \neq 0$. Indeed, by Green's theorem, we have that, for every $i$,
$$
\lambda_1\int_{M'_i}u
=-\int_{M'_i} \Delta u=
-\int_{\partial M'_i}\partial_{\boldsymbol n} u
=
\hat{u}_{\mathcal{C}_{i-1}}-\hat{u}_{\mathcal{C}_i},
$$
where we use the convention $\mathcal{C}_{0}=\mathcal{C}_{2n}$. Assume instead that $\int_{M'_{i}} u = 0$ for every $i$. Then, we have
$$ \mu\int_{M_i'}|u|^2 
\leq
\int_{M_i'}|\nabla u|^2,$$
for every $i \leq 2n$. This is because, as $M'$ is connected and therefore the eigenvalue $0$ has multiplicity $1$, we have by Rayleigh's theorem (\cite[p. 16]{chavel}) that
$$
\mu=
\inf_{\substack{f\in H^1(M'),\ f\ne0\\
                 \int_{M'}f=0}}
\frac{\int_{M'}|\nabla f|^2}
     {\int_{M'}|f|^2}.
$$

Summing over $i$, we deduce that
$$
\mu \int_{M_{2n}}|u|^2 = \sum_{i \leq 2n} \mu\int_{M_i'}|u|^2 
\leq
\sum_{i \leq 2n} \int_{M_i'}|\nabla u|^2
=
\int_{M_{2n}}|\nabla u|^2= \lambda_1\left(M_{2n}\right)\int_{M_{2n}}| u|^2.
$$
This contradicts $\lambda_1<\mu$. Hence, there is some $i$ with $\hat{u}_{\mathcal{C}_i} \neq 0$. 
We note that, while $\Gamma_{2n} \mathrel{\unlhd} \Gamma$ and $\Gamma \mathrel{\unlhd} \tilde{\Gamma}$, it is not necessarily the case that $\Gamma_{2n} \mathrel{\unlhd} \tilde{\Gamma}$. To fix this, write
$$ \tilde{\Gamma}_{2n}:=\bigcap_{\gamma\in\tilde{\Gamma}}
\gamma\Gamma_{2n}\gamma^{-1}  \mathrel{\unlhd} \tilde{\Gamma}.
$$
It is easy to see that $u$ lifts to an eigenfunction  $\tilde{u}$ for $\tilde{\Gamma}_{2n} \backslash \mathbb{H}$, with corresponding eigenvalue $\lambda_1(M_{2n})$. Furthermore, $\tilde{h}=h_i^m$ will be a hyperbolic primitive element of $\tilde{\Gamma}_{2n}$ for some $m$, and the corresponding period will be multiplied by $m$ and a normalisation constant $d$.  In particular, the period is still nonvanishing. We write $\mathcal{C}_i^{(m)}$ for the lift of $\mathcal{C}_i$.
As $\tilde{h}$ is hyperbolic, there is some $g \in \psl$ with $h^{*}=g\tilde{h}g^{-1}$ being diagonal. Then, $l^*:=g \mathcal{C}_i^{(m)}$ will lift to the imaginary axis. We write $\Gamma^{*}_{2n}=g \tilde{\Gamma}_{2n} g^{-1}.$ It is easy to see that $\Gamma_{2n}^{*}$ is still cocompact, and contains no elements with both diagonal entries equal to $0$. We also note that $\tilde{\Gamma}_{2n}$ and $\Gamma_{2n}^{*}$ have the same spectrum. In particular, $\lambda_1(M_{2n})$ is still an eigenvalue, with a corresponding eigenfunction $u^*(z)=\tilde{u}(g^{-1}z)$. 
Finally, we note that, for $\boldsymbol{n}$ as in Theorem \ref{refcount}, we have
$$
(\partial_{\boldsymbol{n}}u^*)(g \cdot z)=\partial_{\boldsymbol{n_i}}\tilde{u}(z),
$$
and, therefore, 
$$md \cdot \hat{u}_{\mathcal{C}_i}=\int_{\mathcal{C}_i^{(m)}}(\partial_{\boldsymbol{n}_i}\tilde{u})(z) \, ds=\int_{g\mathcal{C}_i^{(m)}}(\partial_{\boldsymbol{n}_i}\tilde{u})(g^{-1}z) \, ds=\int_{g\mathcal{C}_i^{(m)}}(\partial_{\boldsymbol{n}}u^{*})(z) \, ds=\lambda_1^{1/2}\hat{u}^*_{1}.$$
Hence, we have $\hat{u}^*_{1}=md\lambda_1^{-1/2}\hat{u}_{\mathcal{C}_i} \neq 0$. 
Take $\lambda=s(1-s)$ to be the smallest nontrivial eigenvalue for $\Gamma^*_{2n}$ with at least one corresponding $u$ with nonvanishing period. Clearly, $s \geq s_1>1-\epsilon$. Therefore, for $\Gamma(\epsilon)=\Gamma^*_{2n}$, we have
$N_{\Delta}(X) \sim -kX^{1-\delta},$ where $\delta=1-s$ and $k=c(s) \sum_{\lambda_j=\lambda}(\hat{u}_{1,j})^2>0.$ This concludes the proof of Theorem \ref{bias}.

\section*{Acknowledgements}
The author was supported by the MTA--HUN-REN RI Lend\"{u}let Automorphic Research Group.

\section*{Statement regarding Artificial Intelligence}
The author used OpenAI’s ChatGPT (GPT-6 Astra) to discuss and critically review portions of the mathematical
arguments, check calculations, assist with locating references, and improve the exposition. The author independently verified all the calculations, proof steps and references included in the final manuscript and takes full responsibility for its contents.

\printbibliography
\end{document}